\documentclass{amsart}
\newtheorem{theo}{Theorem}[section]
\newtheorem{defn}[theo]{Definition}

\newtheorem{lem}[theo]{Lemma}
\newtheorem{lemma}[theo]{Lemma}
\newtheorem{prop}[theo]{Proposition}
\newtheorem{coro}[theo]{Corollary}

\def\coker{\mathop{\mathrm{coker}}\nolimits}
\newcommand{\on}[2]{\setbox0=\hbox{$#1$}\setbox1=\hbox{$#2$}%
            \dimen0=\wd0\advance\dimen0 by \wd1\divide\dimen0 by 2%
             \ifdim\wd0>\wd1$#1$\hskip-\dimen0$#2$\advance\dimen0 by -\wd1%
              \else$#2$\hskip-\dimen0$#1$\advance\dimen0 by -\wd0%
             \fi%
            \hskip\dimen0}
\newcommand{\onn}[2]{\makebox{\on{#1}{#2}}}
\newcommand{\cross}{\smash\times\vphantom\bullet}
\newcommand{\leftblob}[1]{\hspace*{-1ex}\onn{\raisebox{-.5ex}
     {$\genfrac{}{}{0pt}{}{\hfill\hrulefill}{\hphantom{\hspace*{2ex}#1}}$}}
                {\stackrel{#1}{\bullet}}}
\newcommand{\rightblob}[1]{\onn{\raisebox{-.5ex}
     {$\genfrac{}{}{0pt}{}{\hrulefill\hfill}{\hphantom{\hspace*{2ex}#1}}$}}
                 {\stackrel{#1}{\bullet}}\hspace*{-1ex}}
\newcommand{\middleblob}[1]{\onn{\raisebox{-.5ex}
     {$\genfrac{}{}{0pt}{}{\hrulefill}{\hphantom{\hspace*{2ex}#1}}$}}
                 {\stackrel{#1}{\bullet}}}
\newcommand{\leftcross}[1]{\hspace*{-1ex}\onn{\raisebox{-.5ex}
     {$\genfrac{}{}{0pt}{}{\hfill\hrulefill}{\hphantom{\hspace*{2ex}#1}}$}}
                {\stackrel{#1}{\cross}}}
\newcommand{\middlecross}[1]{\onn{\raisebox{-.5ex}
     {$\genfrac{}{}{0pt}{}{\hrulefill}{\hphantom{\hspace*{2ex}#1}}$}}
                 {\stackrel{#1}{\cross}}}

\newcommand{\xoo}[3]{\leftcross{#1}\hspace*{-1ex}
               \middleblob{#2}\hspace*{-1ex}\rightblob{#3}}
\newcommand{\xxo}[3]{\leftcross{#1}\hspace*{-1ex}
               \middlecross{#2}\hspace*{-1ex}\rightblob{#3}}
\newcommand{\oxo}[3]{\leftblob{#1}\hspace*{-1ex}
               \middlecross{#2}\hspace*{-1ex}\rightblob{#3}}
\begin{document}
\title{Twistor theory and the Harmonic Hull}
\author{Michael Eastwood}
\address{Mathematical Sciences Institute, Australian National 
University, ACT 0200,\newline Australia}
\email{meastwoo@member.ams.org}
\author{Feng Xu}
\address{Mathematical Sciences Institute, Australian National 
University, ACT 0200,\newline Australia}
\email{feng.xu@anu.edu.au}
\begin{abstract}
We use twistor theory to identify the harmonic hull of an arbitrary connected 
open subset $U$ of ${\mathbf{R}}^{2m}$ for $m\geq 2$. It is the natural domain 
of analytic continuation in ${\mathbf{C}}^{2m}$ for harmonic functions on~$U$.
\end{abstract}
\dedicatory{Dedicated to Joseph Wolf on the occasion of his 75${}^{th}$ 
birthday.}
\renewcommand{\subjclassname}{\textup{2010} Mathematics Subject Classification}
\maketitle
\renewcommand{\thefootnote}{}
\footnotetext{Eastwood is supported by the Australian Research Council.}
\maketitle

\section*{Introduction}
Let ${\mathbf{R}}^n$ denote the standard $n$-dimensional Euclidean space and 
use $(x_1,\cdots, x_n)$ for its standard co\"ordinates. We shall consider 
${\mathbf{R}}^n$ as the real part of the standard $n$-dimensional complex space
${\mathbf{C}}^n$ with co\"ordinates $(z_1,\cdots, z_n)$ and write 
$\langle z, w \rangle$ for the bilinear form
\[\langle z, w\rangle =\sum_{i=1}^nz_iw_i\]
on ${\mathbf{C}}^n$ extending the usual Euclidean inner product 
on~${\mathbf{R}}^n$. Associated to this inner product, the 
Laplace operator on ${\mathbf{R}}^n$ is
\begin{equation}\label{defofLaplacian}
\Delta=\sum_{i=1}^n\frac{\partial^2}{\partial x_i{}^2}\end{equation}
and it naturally extends to a differential operator on ${\mathbf{C}}^n$, still
denoted by $\Delta$, as
\[\Delta=\sum_{i=1}^n\frac{\partial^2}{\partial z_i{}^2}.\]

Suppose $U$ is a connected open subset of ${\mathbf{R}}^n$ and $u$ is a 
harmonic function on~$U$, i.e.~a solution of Laplace's equation
\[\Delta u=0.\]
It is well-known that such a $u$ is real-analytic and hence has a holomorphic
extension to a open subset of ${\mathbf{C}}^n$ containing $U$. Evidently, this
holomorphic extension satisfies the complex Laplace equation. It is natural to
ask whether there is a common open subset of ${\mathbf{C}}^n$ to which all
harmonic functions on $U$ extend and, if so, whether there is maximal connected
open subset with this property. If $n\geq 4$ is even, these questions both have
affirmative answers and the resulting maximal connected open subset of
${\mathbf{C}}^n$ is called the {\it harmonic hull\/} of $U$ (in~\cite{ACL},
although the resulting set turns out to be the same, the harmonic hull is
initially defined in terms of polyharmonic functions). To identify the harmonic
hull we need some notation.

\begin{defn}[\cite{ACL}, p.~40]\label{isotropic}
For any $z\in{\mathbf{C}}^n$, the isotropic cone through $z$ is 
\[V(z)=\{w\in {\mathbf{C}}^n: \langle w-z, w-z\rangle =0\}.\]
\end{defn}
It is clear that $V$ has a symmetry property, namely 
$z\in V(w)\Longleftrightarrow w \in V(z)$.
\begin{defn}[\cite{ACL}, p.~42]\label{tildeU}
For $U$ a connected open subset of~${\mathbf{R}}^n$, we define 
$\tilde{U}$ to be the connected component of 
\[{\mathbf{C}}^n\setminus \bigcup_{x\in {\mathbf{R}}^n\setminus U}V(x)\]
containing~$U$.
\end{defn}
By the symmetry property of $V$, note that
\begin{eqnarray}\label{back_propagate}
{\mathbf{C}}^n\setminus \bigcup_{x\in \mathbf{R}^n\setminus U}V(x)\quad=\quad
\{z\in{\mathbf{C}}^n\mbox{ s.t.\ }V(z)\cap{\mathbf{R}}^n\subset U\}.
\end{eqnarray}
The purpose of this article is to prove the following result using twistor
theory (in \cite{ACL} it is proved by different means under some further 
conditions on~$U$).
\begin{theo}\label{HarmonicHull}
Any connected open subset $U\subseteq {\mathbf{R}}^{2m}$ for $m\geq 2$ has a 
harmonic hull and it is given by $\tilde{U}$.
\end{theo}
It is clear that $\tilde{U}$ is maximal because for 
$x\in{\mathbf{R}}^{2m}\setminus U$ the Newtonian potential centred at $x$
\[r_x(z)=\frac{1}{\langle z-x,z-x\rangle^{m-1}}\]
is harmonic in $\tilde{U}$ but cannot be extended through~$V(x)$. Therefore, it
remains to show that every harmonic function on $U$ indeed extends
to~$\tilde{U}$.

In dimension 2 one needs to suppose that $U$ is simply connected, in which case
the result is easily derived by complex analysis. One only needs to know that
every harmonic function on $U$ can be written as the real part of a holomorphic
function, a fact that can be viewed as a rather degenerate form of twistor
theory. In dimension 4 twistor theory comes to the fore, providing a
replacement for this 2-dimensional statement. We shall begin with Bateman's
formula~\cite{bateman} for harmonic functions of four variables. This formula
is not completely precise and requires careful interpretation for rigorous
application. Twistor theory will be used to interpret the formula as a
transform on suitable cohomology. This is the classical Penrose transform
\cite{holyoke92,hitchin} and is sufficient to prove Theorem~\ref{HarmonicHull}.
This treatment in dimension 4 allows for generalizations to higher even
dimensions, using~\cite{murray}. We shall see in Section~\ref{odd} that the
harmonic hull in odd dimensions behaves differently.

Although the results in this article are rather straightforward deductions from
the Penrose transform, as a by-product we clarify the statements in~\cite{ACL}.
In particular, we draw attention to the sharp distinction between even and odd
dimensions. This is quite natural from the twistor theory point of view.
Another motivation, however, is to present the twistor approach as the natural
method that we anticipate will extend to other settings. In particular, our
motivation comes from a recent article \cite{ks} by Kroetz and Schlichtkrull
who use techniques from partial differential equations to show that
eigenfunctions of the Laplacian on a Riemannian symmetric space extend
holomorphically to its complex crown and we suggest that the integral
transforms on cohomology discussed in~\cite{fhw} might be used to the same
effect.

The authors would like to thank Henrik Schlichtkrull for drawing their
attention to the harmonic hull, Alexander Isaev for useful discussions on
constructible sets, and Amnon Neeman for simplifying the proof of
Lemma~\ref{PQPclosed}. Thanks are also due to the referee for spotting several
crucial misprints.

Finally, the first author would like to express thanks to Joe Wolf for
innumerable and inspirational mathematical conversations over the past twenty
years. This article is motivated by Joe's extensive work on the double
fibration transform and we are pleased to dedicate it to him on the occasion of
his $75^{\mathrm{th}}$ birthday.

\section{Harmonic hull in dimension~2}
In dimension 2, the complex Newtonian potential
\begin{eqnarray}\label{newt}r_x(z)=\log\langle z-x,z-x\rangle\end{eqnarray}
is harmonic but is only well-defined locally once a branch of logarithm has
been chosen. Consequently, an annulus in ${\mathbf{R}}^2$ does not have a
well-defined harmonic hull in~${\mathbf{C}}^2$ (for consider trying to extend
the real Newtonian potential based at a point encircled by the annulus). If $U$
is simply connected, however, then the result is exactly as in
Theorem~\ref{HarmonicHull}. Firstly, the potentials (\ref{newt}) for
$x\in{\mathbf{R}}^2\setminus U$ show that $\tilde{U}$ is maximal. Secondly, to
show that harmonic functions on $U$ extend to $\tilde{U}$ we may proceed as
follows.

It is well known that a harmonic function $u$ on a simply connected open
subset $U\subseteq {\mathbf{R}}^2$ can be written as
\begin{eqnarray}\label{HarmRepDim2}
u=f(\zeta)+g(\bar{\zeta}),\enskip \zeta=x_1+ix_2\end{eqnarray} 
for holomorphic functions $f$ and $g$. This representation evidently extends to
\begin{eqnarray}\label{extend}
\tilde{u}(z_1,z_2)=f(z_1+iz_2)+g(z_1-iz_2)\end{eqnarray}
whenever the right hand side makes sense, i.e.\ precisely when
\begin{eqnarray}\label{precise}
z_1+iz_2\in U\enskip\mbox{and}\enskip z_1-iz_2\in \bar{U},\end{eqnarray}
where $\bar{U}$ denotes the set of complex conjugates of points in~$U$. On the 
other hand, we may compute $\tilde{U}$ from the right hand side 
of~(\ref{back_propagate}). Specifically, $V(z)\cap{\mathbf{R}}^2$ is the set
\[\begin{array}{l}
\phantom{={}}\{(x_1,x_2)\in{\mathbf{R}}^2\mbox{ s.t.\ }
(z_1-x_1)^2+(z_2-x_2)^2=0\}\\
=\{(x_1,x_2)\in{\mathbf{R}}^2\mbox{ s.t.\ }
((z_1-x_1)+i(z_2-x_2))((z_1-x_1)-i(z_2-x_2))=0\}\\
=\{(x_1,x_2)\in{\mathbf{R}}^2\mbox{ s.t.\ }
z_1+iz_2=x_1+ix_2\mbox{ or }z_1-iz_2=x_1-ix_2\}
\end{array}\]
and so $V(z)\cap{\mathbf{R}}^2\subset U$ if and only if the conditions
(\ref{precise}) hold. It is also clear from (\ref{precise}) that these
conditions define a connected (and simply-connected) subset of
${\mathbf{C}}^2$. Therefore (\ref{extend}) extends $u$ to $\tilde{U}$, as
required.
    
\section{Harmonic hull in dimension 4}
The investigation in dimension 2 suggests that a representation of harmonic
functions by holomorphic data will also be useful in understanding the harmonic
hull in higher dimensions. In dimension 4 such a representation, albeit too
na\"{\i}ve for our purposes, is given by Bateman's formula~\cite{bateman}.
\subsection{Bateman's formula}
Let $f$ be a holomorphic function of 3 complex variables. 
Consider the function defined on $U\subseteq{\mathbf{R}}^4$ by
\begin{eqnarray}\label{Bateman}
u(x)=\oint_{\gamma} f\left((x_1+ix_2)+(ix_3+x_4)\zeta, 
(ix_3-x_4)+(x_1-ix_2)\zeta,\zeta \right) d\zeta,
\end{eqnarray}
where $\gamma$ is some contour on the complex $\zeta$-plane. Differentiating
under the integral sign shows that $u$ satisfies the Laplace equation.

The cautious reader may be concerned, however, that the domain of definition
for $f$ has not been specified nor has the precise location of the
contour~$\gamma$. If $f$ is defined on all of ${\mathbf{C}}^4$, for example,
then $u(x)$ will be identically zero by Cauchy's theorem. Precision will be
restored later by twistor theory. For the moment, let us pretend that this
expression makes good unambiguous sense and let us further assume that every
harmonic function $u$ has such a representation, a fact also to be justified by
twistor theory. To see how Bateman's formula (\ref{Bateman}) allows us to
extend $u$, observe that the mapping implicit in the integrand 
of~(\ref{Bateman}), namely 
\[\zeta\mapsto
\left((x_1+ix_2)+(ix_3+x_4)\zeta, (ix_3-x_4)+(x_1-ix_2)\zeta,\zeta \right),\]
defines, for each $x=(x_1,x_2,x_3,x_4)\in {\mathbf{R}}^4$, a complex affine
line $L_x\subset{\mathbf{C}}^3$. The same assignment equally defines for each
$z=(z_1,z_2,z_3,z_4)\in{\mathbf{C}}^4$, a line $L_z\subset{\mathbf{C}}^3$. Two
such lines $L_z$ and $L_{z'}$ intersect if and only if the linear system 
in~$\zeta$
\begin{eqnarray*}
(z_1-z_1')+i(z_2-z_2')+(i(z_3-z_3')+z_4-z_4')\zeta=0,\\
i(z_3-z_3')-(z_4-z_4')+((z_1-z_1')-i(z_2-z_2'))\zeta=0
\end{eqnarray*}
has a solution. Thus, a necessary condition for 
$L_z\cap L_{z'}\not=\emptyset$ is that the determinant of the coefficient 
matrix vanish, more specifically
\begin{eqnarray}\label{NullCondition}\langle z-z',z-z'\rangle=0.\end{eqnarray}
This condition is not sufficient but if we embed ${\mathbf{C}}^3$ into 
projective space ${\mathbf{CP}}^3$ so that the third homogeneous co\"ordinate 
is~$1$, then its composition with the previous mapping gives
\[{\mathbf{C}}\ni\zeta\mapsto 
\left[(x_1+ix_2)+(ix_3+x_4)\zeta, (ix_3-x_4)+(x_1-ix_2)\zeta,1,\zeta\right]
\in{\mathbf{CP}}^3,\]
which naturally compactifies as an embedding of the projective line 
${\mathbf{CP}}^1\hookrightarrow{\mathbf{CP}}^3$ so that (\ref{NullCondition}) 
is now sufficient for non-trivial intersection. In other words, if we now 
write $L_z$ for the image in ${\mathbf{CP}}^3$ of the embedding 
${\mathbf{CP}}^1\hookrightarrow{\mathbf{CP}}^3$ given by
\begin{equation}\label{embed}[\zeta_1,\zeta_2]\mapsto 
\left[(z_1+iz_2)\zeta_1+(iz_3+z_4)\zeta_2, 
(iz_3-z_4)\zeta_1+(z_1-iz_2)\zeta_2,\zeta_1,\zeta_2\right],\end{equation}
then $L_z\cap L_{z'}\not=\emptyset$ if and only if (\ref{NullCondition}) holds.
    
In particular, if $x$ and $x'$ are distinct real points, then
(\ref{NullCondition}) never holds and so $L_x$ and $L_{x'}$ can never
intersect. In fact, it is easy to check that the set of lines $\{L_x:x\in U\}$
foliates an open subset ${\mathcal{U}}\subset{\mathbf{CP}}^3$. Certainly, if
one allows $x$ in Bateman's formula to become complex, i.e.\ we consider
(\ref{Bateman}) with $x\in{\mathbf{R}}^4$ simply replaced by
$z\in{\mathbf{C}}^4$, then we obtain a holomorphic solution of Laplace's
equation, say~$\tilde{u}(z)$, extending~$u(x)$. Thus, if we could make good
sense of Bateman's formula as associating a harmonic function to some
holomorphic data on ${\mathcal{U}}$ and if every harmonic function of $U$ were
to arise in this way, then we would expect the same formula to associate an
extension $\tilde{u}(z)$ of $u(x)$ provided that $L_z$ were contained
in~${\mathcal{U}}$. Therefore, we should identify 
$\{z\in{\mathbf{C}}^4\mbox{ s.t.\ }L_z\subset{\mathcal{U}}\}$. To do this we 
observe that the region of ${\mathbf{CP}}^3$ swept out by $L_x$ for 
$x\in{\mathbf{R}}^4$ is the same as the region swept by $L_z$ for 
$z\in{\mathbf{C}}^4$ (only the line $[*,*,0,0]\in{\mathbf{CP}}^3$ `at 
infinity' is omitted in either case). Therefore, to say that 
$L_z\subset{\mathcal{U}}$ is to say that $L_z$ does not intersect $L_x$ for 
all $x\in{\mathbf{R}}^4\setminus U$ and by (\ref{NullCondition}) this is to 
say that $\langle z-x,z-x\rangle\not=0$. In other words, in terms of 
Definition~\ref{isotropic}, we have
\[\{z\in{\mathbf{C}}^4\mbox{ s.t.\ }L_z\subset{\mathcal{U}}\}=
\{z\in{\mathbf{C}}^4\mbox{ s.t.\ }z\not\in V(x)\;\forall 
x\in{\mathbf{R}}^4\setminus U\}.\]
According to Definition~\ref{tildeU} we conclude that harmonic functions on $U$
extend to~$\tilde{U}$, as required. Of course, this reasoning is based solely
on the geometry implicit in the form of the integrand in Bateman's
formula~(\ref{Bateman}). Once we use this geometry to make rigorous sense of
Bateman's formula, then we shall have a genuine proof.

\subsection{Justification of Bateman's formula: the Penrose transform}
Let us elaborate on the geometry uncovered in the previous section. We
associated a complex line $L_x$ in ${\mathbf{CP}}^3$ for each
$x\in{\mathbf{R}}^4$. A complex line in ${\mathbf{CP}}^3$ is the same as a
complex $2$-dimensional linear subspace of~${\mathbf{C}}^4$. Thus, we obtain an
embedding ${\mathbf{R}}^4\hookrightarrow{\mathrm{Gr}}_2({\mathbf{C}}^4)$, where
${\mathrm{Gr}}_2({\mathbf{C}}^4)$ denotes the Grassmannian of $2$-dimensional 
linear subspaces in~${\mathbf{C}}^4$. Specifically,
\[(x_1,x_2,x_3,x_4)\mapsto\left\{(Z_1,Z_2,Z_3,Z_4)\mbox{ s.t.\ }
\left[\!\begin{array}cZ_1\\ Z_2\end{array}\!\right]=
\left[\!\begin{array}{cc}x_1+ix_2&ix_3+x_4\\ix_3-x_4&x_1-ix_2
\end{array}\!\right]
\left[\!\begin{array}cZ_3\\ Z_4\end{array}\!\right]\right\}.\]
Now consider the double fibration
\begin{eqnarray}\label{twistorcorrespondence}
\raisebox{-25pt}{\begin{picture}(80,55)(0,-5)
\put(40,40){\makebox(0,0){${\mathbf{F}}_{1,2}({\mathbf{C}}^4)$}}
\put(0,0){\makebox(0,0){${\mathbf{CP}}^3$}}
\put(80,0){\makebox(0,0){${\mathrm{Gr}}_2({\mathbf{C}}^4)$}}
\put(30,30){\vector(-1,-1){20}}
\put(50,30){\vector(1,-1){20}}
\put(14,21){\makebox(0,0){$\mu$}}
\put(67,21){\makebox(0,0){$\nu$}}
\end{picture}}\end{eqnarray}
where ${\mathbf{F}}_{1,2}({\mathbf{C}}^4)$ denotes the complex flag manifold
\[{\mathbf{F}}_{1,2}({\mathbf{C}}^4)=
\{L_1\subset L_2\subset{\mathbf{C}}^4\mbox{ where }\dim L_i=i\}\]
and where $\mu$ and $\nu$ are the tautological `forgetful' maps. The formula
(\ref{embed}) for $L_z$ is now interpreted as $L_z\equiv\mu(\nu^{-1}(z))$ for
$z\in{\mathbf{C}}^4\cong{\mathbf{C}}^{2\times 2}\hookrightarrow
{\mathrm{Gr}}_2({\mathbf{C}}^4)$ a standard affine co\"ordinate patch (with a
convenient change of basis included in 
${\mathbf{C}}^4\cong{\mathbf{C}}^{2\times 2}$). 
We now ask how the embedding 
${\mathbf{R}}^4\hookrightarrow{\mathrm{Gr}}_2({\mathbf{C}}^4)$ sits with 
respect to the double fibration~(\ref{twistorcorrespondence}). 
\begin{prop}\label{manufacture_tau}\mbox{}    
\begin{itemize}
\item The closure of\/
${\mathbf{R}}^4\hookrightarrow{\mathrm{Gr}}_2({\mathbf{C}}^4)$ is a smooth 
embedding $S^4\hookrightarrow{\mathrm{Gr}}_2({\mathbf{C}}^4)$.
\item For all $Z\in{\mathbf{CP}}^3$, the intersection 
$\nu(\mu^{-1}(Z))\cap S^4$ is a single point.
\item The assignment $Z\mapsto\nu(\mu^{-1}(Z))\cap S^4$ defines a fibration
$\tau:{\mathbf{CP}}^3\to S^4$.
\end{itemize}
\end{prop}
\begin{proof} The first point could be checked by looking in all standard 
affine co\"ordinate patches. It is convenient, however, to adopt a viewpoint 
that generalises to higher dimensions as follows.
Regarding ${\mathrm{Gr}}_2({\mathbf{C}}^4)$ as the simple $2$-forms in $4$ 
variables up to scale, the Pl\"ucker embedding
\[{\mathrm{Gr}}_2({\mathbf{C}}^4)=\{[\alpha\wedge\beta]\}=
\{[v]\mbox{ s.t.\ }v\wedge v=0\}\hookrightarrow
{\mathbf{P}}(\Lambda^2{\mathbf{C}}^4)={\mathbf{CP}}^5\]
identifies ${\mathrm{Gr}}_2({\mathbf{C}}^4)$ as the non-singular quadric 
${\mathbf{Q}}_4$ of dimension~$4$. Then, for the composition
${\mathbf{R}}^4\hookrightarrow{\mathrm{Gr}}_2({\mathbf{C}}^4)\hookrightarrow
{\mathbf{CP}}^5$ 
we obtain
\[\begin{array}{crl}(x_1,x_2,x_3,x_4)&\mapsto&
[(x_1+ix_2,ix_3-x_4,1,0)\wedge(ix_3+x_4,x_1-ix_2,0,1)]\\[3pt]
&=&\left[\sum_{i<j}\phi_{ij}dZ_i\wedge dZ_j\right],
\end{array}\]
where
\[\begin{array}{rclrclrcl}
\phi_{12}&=&x_1{}^2+x_2{}^2+x_3{}^2+x_4{}^2\enskip
&\phi_{13}&=&-ix_3-x_4\enskip
&\phi_{14}&=&x_1+ix_2\\
\phi_{23}&=&-x_1+ix_2\enskip
&\phi_{24}&=&ix_3-x_4\enskip
&\phi_{34}&=&1,
\end{array}\]
the quadric ${\mathbf{Q}}_4$ being given by
$\phi_{12}\phi_{34}-\phi_{13}\phi_{24}+\phi_{14}\phi_{23}=0$. Now consider the
embedding ${\mathbf{RP}}^5\hookrightarrow{\mathbf{CP}}^5$ given by
\[
[\xi_0,\xi_1,\xi_2,\xi_3,\xi_4,\xi_5]\mapsto
[\xi_0-\xi_5,-i\xi_3-\xi_4,\xi_1+i\xi_2,-\xi_1+i\xi_2,i\xi_3-\xi_4,\xi_0+\xi_5]
\]
and notice that 
\[
(\phi_{12}\phi_{34}-\phi_{13}\phi_{24}+\phi_{14}\phi_{23})|_{{\mathbf{RP}}^5}
=\xi_0{}^2-\xi_1{}^2-\xi_2{}^2-\xi_3{}^2-\xi_4{}^2-\xi_5{}^2.\]
It follows that the closure of 
${\mathbf{R}}^4\hookrightarrow{\mathrm{Gr}}_2({\mathbf{C}}^4)$ is the 
intersection ${\mathbf{RP}}^5\cap{\mathbf{Q}}_4$ and that this in turn is
\[\{[\xi_0,\xi_1,\xi_2,\xi_3,\xi_4,\xi_5]\in{\mathbf{RP}}^5\mbox{ s.t.\ }
\xi_1{}^2+\xi_2{}^2+\xi_3{}^2+\xi_4{}^2+\xi_5{}^2=\xi_0{}^2\},\]
which may be identified with the sphere $S^4$. For later use, let us note that
the induced mapping ${\mathbf{R}}^4\hookrightarrow S^4$ is given by
\[{\mathbf{R}}^4\ni{\mathbf{x}}\mapsto\frac{1}{1+\|{\mathbf{x}}\|^2}
\left[\begin{array}c 2{\mathbf{x}}\\ 1-\|{\mathbf{x}}\|^2\end{array}\right]
\in S^4\subset{\mathbf{R}}^5,\]
which is inverse stereographic projection. If we define 
$\theta:{\mathbf{CP}}^3\to{\mathbf{CP}}^3$ by 
\[\theta[Z_1,Z_2,Z_3,Z_4]=[-\bar{Z}_2,\bar{Z}_1,-\bar{Z}_4,\bar{Z}_3]\]
then the plane $\phi\equiv[Z\wedge\theta Z]$ is given by 
\[\begin{array}{rclrclrcl}
\phi_{12}&=&Z_1\bar{Z}_1+Z_2\bar{Z}_2\enskip
&\phi_{13}&=&-Z_1\bar{Z}_4+Z_3\bar{Z}_2\enskip
&\phi_{14}&=&Z_1\bar{Z}_3+Z_4\bar{Z}_2\\
\phi_{23}&=&-Z_2\bar{Z}_4-Z_3\bar{Z}_1\enskip
&\phi_{24}&=&Z_2\bar{Z}_3-Z_4\bar{Z}_1\enskip
&\phi_{34}&=&Z_3\bar{Z}_3+Z_4\bar{Z}_4
\end{array}\]
and satisfies $\bar{\phi}_{12}=\phi_{12}$, $\bar{\phi}_{13}=\phi_{24}$, 
$\bar{\phi}_{14}=-\phi_{23}$, $\bar{\phi}_{34}=\phi_{34}$ which are exactly 
the conditions to lie in ${\mathbf{RP}}^5\hookrightarrow{\mathbf{CP}}^5$ as 
defined above. It is straightforward to check that this is the only plane 
through $Z$ with this property. Finally, if we solve for 
$[\xi_0,\xi_1,\xi_2,\xi_3,\xi_4,\xi_5]$ we find that 
$\tau:{\mathbf{CP}}^3\to S^4$ is given by
\[\left[\begin{array}cZ_1\\ Z_2\\ Z_3\\ Z_4\end{array}\right]\longmapsto
\frac{1}{Z_1\bar{Z}_1+Z_2\bar{Z}_2+Z_3\bar{Z}_3+Z_4\bar{Z}_4}
\left[\begin{array}c
   Z_1\bar{Z}_3+Z_2\bar{Z}_4+Z_3\bar{Z}_1+Z_4\bar{Z}_2\\
i(-Z_1\bar{Z}_3+Z_2\bar{Z}_4+Z_3\bar{Z}_1-Z_4\bar{Z}_2)\\
i(-Z_1\bar{Z}_4-Z_2\bar{Z}_3+Z_3\bar{Z}_2+Z_4\bar{Z}_1)\\
   Z_1\bar{Z}_4-Z_2\bar{Z}_3-Z_3\bar{Z}_2+Z_4\bar{Z}_1\\
-Z_1\bar{Z}_1-Z_2\bar{Z}_2+Z_3\bar{Z}_3+Z_4\bar{Z}_4\end{array}\right]\]
which is certainly a submersion.
\end{proof}

The following theorem provides a strict interpretation of Bateman's formula in
which the formula itself is viewed as an attempt to write ${\mathcal{P}}$ in
\v{C}ech cohomology.

\begin{theo}\label{classicalPenrose}
Suppose $U\subseteq{\mathbf{R}}^4$ is an open subset. There is a 
natural isomorphism
\begin{equation}\label{bateman_justified}
{\mathcal{P}}:H^1(\tau^{-1}(U),{\mathcal{O}}(-2))
\stackrel{\;\simeq\quad}{\longrightarrow}
\{\phi:U\to{\mathbf{C}}\mbox{ s.t.\ }\Delta\phi=0\}.\end{equation}
\end{theo}
\begin{proof} A detailed elementary proof may be found in \cite{split}, for
example. Here, we shall present a more abstract proof avoiding some of the 
detail. Much of the argument is rather general and so we shall generalise the 
notation, writing 
\begin{equation}\label{realdiagram}
\raisebox{-20pt}{\begin{picture}(70,50)
\put(35,40){\makebox(0,0){${\mathfrak{X}}$}}
\put(0,5){\makebox(0,0){$Z$}}
\put(70,5){\makebox(0,0){${\mathbf{C}}M$}}
\put(30,35){\vector(-1,-1){25}}
\put(40,35){\vector(1,-1){25}}
\put(107,5){\makebox(0,0){$M$}}
\put(90,5){\makebox(0,0){$\supset$}}
\put(10,20){\makebox(0,0){$\mu$}}
\put(61,20){\makebox(0,0){$\nu$}}
\end{picture}}
\end{equation}
instead of (\ref{twistorcorrespondence}) for a general double fibration of
complex manifolds, where we are assuming and incorporating into the notation
that the complex manifold ${\mathbf{C}}M$ is the complexification of a smooth
real manifold~$M$ (in fact, we shall only need ${\mathbf{C}}M$ in a
neighbourhood of~$M$). In our case
${\mathbf{C}}M={\mathrm{Gr}}_2({\mathbf{C}}^4)$ is the complexification of
$M=S^4$ and here we know that for each $Z\in{\mathbf{CP}}^3$, there is a unique
point in~$S^4$, namely~$\tau(Z)$, such that $Z$ and $\tau(Z)$ are in
correspondence under~(\ref{twistorcorrespondence}). We may abstract this
knowledge by adding to our diagram as follows
\begin{equation}\label{specialrealdiagram}
\raisebox{-20pt}{\begin{picture}(70,50)
\put(35,40){\makebox(0,0){${\mathfrak{X}}$}}
\put(0,5){\makebox(0,0){$Z$}}
\put(70,5){\makebox(0,0){${\mathbf{C}}M$}}
\put(30,35){\vector(-1,-1){25}}
\put(40,35){\vector(1,-1){25}}
\put(75,35){\vector(1,-1){25}}
\put(72,40){\makebox(0,0){$Z$}}
\put(107,5){\makebox(0,0){$M$}}
\put(53.5,40){\makebox(0,0){$\supset$}}
\put(90,5){\makebox(0,0){$\supset$}}
\put(67,35){\line(-2,-1){17}}
\put(66.5,36){\line(-2,-1){17}}
\put(7,5){\line(2,1){41}}
\put(6.5,6){\line(2,1){41}}
\put(10,20){\makebox(0,0){$\mu$}}
\put(61,20){\makebox(0,0){$\nu$}}
\put(96,20){\makebox(0,0){$\tau$}}
\end{picture}}
\end{equation}
where $Z$ is realised as a smooth submanifold of the complex manifold
${\mathfrak{X}}$ in addition to being its holomorphic quotient under~$\mu$ and
$\tau$ is simply the restriction of the holomorphic fibration $\nu$ to $Z$
realised in this way. In general, let us write
\[m=\dim_{\mathbf{C}}(\mbox{fibres of~$\mu$})\quad\mbox{and}\quad
s=\dim_{\mathbf{C}}(\mbox{fibres of~$\nu$}).\]
In our case (\ref{twistorcorrespondence}) we have $m=2$ and $s=1$. In general,
it follows that
\begin{equation}\label{dimensions}
\dim_{\mathbf{R}}M=\dim_{\mathbf{C}}{\mathbf{C}}M=2m\qquad
\dim_{\mathbf{C}}Z=m+s\qquad\dim_{\mathbf{C}}{\mathfrak{X}}=2m+s.\end{equation}
Although $\tau$ is not a holomorphic mapping it does have holomorphic fibres.
This implies that if we write $\Lambda_Z^1$ for the bundle of
${\mathbf{C}}$-valued $1$-forms on $Z$, then the subbundle
$\tau^*\Lambda_M^1\hookrightarrow\Lambda_Z^1$ is preserved by the complex
structure and we can decompose the quotient
$\Lambda_\tau^1\equiv\Lambda_Z^1/\tau^*\Lambda_M^1$ into types
$\Lambda_\tau^1=\Lambda_\tau^{1,0}\oplus\Lambda_\tau^{0,1}$, obtaining a
surjective homomorphism of bundles $\Lambda_Z^{0,1}\to\Lambda_\tau^{0,1}$
on~$Z$. We may ask about its kernel, noting that, whatever else is true, the
rank of this kernel is certainly $(m+s)-s=m$. There is a natural complex vector
bundle of rank $m$ on ${\mathfrak{X}}$, namely $\Lambda_\mu^{1,0}$, the bundle
of $(1,0)$-forms along the fibres of~$\mu$. Let us use the same notation for
the restriction of this bundle to $Z\subset{\mathfrak{X}}$ and claim a short
exact sequence
\begin{equation}\label{key}
0\to\Lambda_\mu^{1,0}\to\Lambda_Z^{0,1}\to\Lambda_\tau^{0,1}\to 0
\end{equation}
of complex vector bundles on~$Z$. It is a matter of linear algebra in the
tangent spaces to check the validity of this claim. It is convenient to view 
(\ref{key}) as filtering the bundle $\Lambda_Z^{0,1}$, thereby inducing 
filtrations on all its exterior products~$\Lambda_Z^{0,q}$. For example
\[\Lambda_Z^{0,2}=\Lambda_\tau^{0,2}\;+\;
\Lambda_\tau^{0,1}\!\otimes\!\Lambda_\mu^{1,0}
\;+\;\Lambda_\mu^{2,0},\]
meaning that these are the subquotients listed in the natural order, starting
with the quotient itself. Furthermore, as a consequence of $\Lambda_\mu^{1,0}$ 
being the restriction of a holomorphic vector bundle on~${\mathfrak{X}}$, it 
follows easily that the filtration on $\Lambda^{0,\bullet}$ is compatible with 
the $\bar\partial$-operator. (For example, the composition
\[\Lambda_\mu^{1,0}\to\Lambda_Z^{0,1}\xrightarrow{\,\bar\partial\,}
\Lambda_Z^{0,2}\to\Lambda_\tau^{0,2}\]
vanishes and the induced differential operator 
\[\bar\partial_\tau:\Lambda_\mu^{1,0}\to
\Lambda_\tau^{0,1}\otimes\Lambda_\mu^{1,0}\]
is the full $\bar\partial$-operator on $\Lambda_\mu^{1,0}$ but restricted to
act only along the fibres of~$\tau$.) In principle, when the fibres of $\tau$
are compact, it is now a matter of diagram chasing to relate the analytic
cohomology $H^r(\tau^{-1}(U),{\mathcal{O}})$ computed by means of the Dolbeault
resolution $0\to{\mathcal{O}}\to\Lambda^{0,\bullet}$ to data down on~$U$. This
is what is done in detail in~\cite{split}. In practise, however, it is
convenient to relegate this task to a spectral sequence, namely that of a
filtered complex~\cite{ssguide}. Firstly, one replaces
\begin{itemize}
\item $M$ by any open subset $U\subset M$,
\item $Z$ by $\tau^{-1}(U)$,
\item ${\mathbf{C}}M$ by any neighbourhood ${\mathbf{C}}U$ of $U$ in
${\mathbf{C}}M$,
\item ${\mathfrak{X}}$ by $\nu^{-1}({\mathbf{C}}U)$.
\end{itemize}
The resulting spectral sequence is
\[E_1^{p,q}=\Gamma(U,\tau_*^q\Lambda_\mu^{p,0})\Longrightarrow 
H^{p+q}(\tau^{-1}(U),{\mathcal{O}})\]
where $\tau_*^q\Lambda_\mu^{p,0}$ denotes the smooth vector bundle on $U$ whose
fibre over $u\in U$ is the finite-dimensional Dolbeault cohomology
$H^q(\tau^{-1}(u),\Lambda_\mu^{p,0})$ of the compact complex manifold
$\tau^{-1}(u)$ with coefficients in the holomorphic vector
bundle~$\Lambda_\mu^{p,0}$, assuming that the dimension of these spaces remain
constant as $u\in U$ varies.

The spectral sequence we need is a minor variation on this one. It is obtained 
by incorporating a holomorphic vector bundle $V$ on~$Z$. Certainly, we may 
tensor (\ref{key}) with~$V$, obtaining a short exact sequence that we shall 
write as
\[0\to\Lambda_\mu^{1,0}(V)\to\Lambda_Z^{0,1}(V)\to\Lambda_\tau^{0,1}(V)\to 0.\]
One easily checks that $V$ being holomorphic forces the induced filtration on
$\Lambda^{0,\bullet}(V)$ to be compatible with the coupled
$\bar\partial$-operator. The resulting spectral sequence is
\begin{equation}\label{ssV}
E_1^{p,q}=\Gamma(U,\tau_*^q\Lambda_\mu^{p,0}(V))\Longrightarrow 
H^{p+q}(\tau^{-1}(U),{\mathcal{O}}(V))\end{equation}
where $\tau_*^q\Lambda_\mu^{p,0}(V)$ denotes the smooth vector bundle on $U$
whose fibre over $u\in U$ is the finite-dimensional Dolbeault cohomology
$H^q(\tau^{-1}(u),\Lambda_\mu^{p,0}(V))$ of $\tau^{-1}(u)$ with coefficients in
the holomorphic vector bundle $\Lambda_\mu^{p,0}\otimes V$, assuming that the
dimension of these spaces remain constant as $u\in U$ varies.

Notice that the bundles $\Lambda_\mu^{p,0}\otimes V$ on $Z$ may be seen as the
restriction to $Z\subset{\mathfrak{X}}$ of the holomorphic vector bundles
$\Lambda_\mu^{p,0}\otimes\mu^*V$ on~${\mathfrak{X}}$. Assuming that they are
locally free and hence represent vector bundles, the direct images
$\nu_*^q{\mathcal{O}}(\Lambda_\mu^{p,0}\otimes\mu^*V)$ on ${\mathbf{C}}U$
restrict to $U$ as the smooth vector bundles~$\tau_*^q\Lambda_\mu^{p,0}(V)$. 
In the homogeneous setting, as we shall see, this observation allows us to 
compute the terms in~(\ref{ssV}), the point being that the double fibration 
(\ref{twistorcorrespondence}) is homogeneous under the action of 
${\mathrm{SL}}(4,{\mathbf{C}})$. 

Without further ado, we now switch to the notation of \cite{beastwood} to
compute the direct images
$\tau_*^q{\mathcal{O}}(\Lambda_\mu^{p,0}\otimes\mu^*V)$ as homogeneous vector 
bundles on~${\mathrm{Gr}}_2({\mathbf{C}}^4)$ when $V$ is the line bundle 
corresponding to~${\mathcal{O}}(-2)$. With the notation of~\cite{beastwood}, 
only mildly abused, we have 
\[V={\mathcal{O}}(-2)=\xoo{-2}{0}{0}\qquad\mbox{and}\qquad
\Lambda_\mu^{1,0}=\xxo{1}{-2}{1}\]
whence
\[\begin{array}{ccccc}
\mu^*V&\stackrel{d_\mu}{\longrightarrow}&\Lambda_\mu^{1,0}\otimes\mu^*V&
\stackrel{d_\mu}{\longrightarrow}&\Lambda_\mu^{2,0}\otimes\mu^*V\\
\|&&\|&&\|\\
\xxo{-2}{0}{0}&\longrightarrow&\xxo{-1}{-2}{1}&\longrightarrow&\xxo{0}{-3}{0}
\end{array}\]
the only non-zero direct images of which are
\[\nu_*^1\xxo{-2}{0}{0}=\oxo{0}{-1}{0}\qquad\mbox{and}\qquad
\nu_*^0\:\xxo{0}{-3}{0}=\oxo{0}{-3}{0}\]
(cf.~\cite[pp.~99-100]{beastwood}). Consequently, the spectral sequence 
(\ref{ssV}) reads
\[\raisebox{10pt}{$E_1^{p,q}={}$}\quad\begin{picture}(140,30)
\put(0,0){\vector(1,0){140}}
\put(0,0){\vector(0,1){30}}
\put(30,6){\makebox(0,0){$0$}}
\put(70,6){\makebox(0,0){$0$}}
\put(110,6){\makebox(0,0){$\Gamma(U,{\mathcal{E}}[-3])$}}
\put(30,24){\makebox(0,0){$\Gamma(U,{\mathcal{E}}[-1])$}}
\put(70,24){\makebox(0,0){$0$}}
\put(110,24){\makebox(0,0){$0$}}
\end{picture}\]
where ${\mathcal{E}}[k]$ denotes smooth functions of {\it conformal
weight\/}~$k$ on~$S^4$ as a homogeneous space under the action
of~${\mathrm{SO}}(n+1,1)$ as explained, for example, in~\cite{eg}. It is
well-known (and also explained in~\cite{eg}) that the only
${\mathrm{SO}}(n+1,1)$-invariant linear differential operator
${\mathcal{E}}[-1]\to{\mathcal{E}}[-3]$~is, up to a constant multiple, the
conformal Laplacian and in the flat metric this is the ordinary
Laplacian~(\ref{defofLaplacian}). Provided that the $E_2$-differential
$\Gamma(U,{\mathcal{E}}[-1])\to\Gamma(U,{\mathcal{E}}[-3])$ is non-zero, our
desired conclusion (\ref{bateman_justified}) follows. An easy way to see that
this $E_2$-differential is, indeed, non-zero is to observe the following points
(the first of which, by Peetre's Theorem~\cite{peetre}, also shows why we
know without computation that
$d_2:\Gamma(U,{\mathcal{E}}[-1])\to\Gamma(U,{\mathcal{E}}[-3])$ is induced by a
linear differential operator ${\mathcal{E}}[-1]\to{\mathcal{E}}[-3]$).
\begin{itemize}
\item The spectral sequence (\ref{ssV}) respects restriction to smaller open
subsets.
\item $H^2(\tau^{-1}(U),{\mathcal{O}}(-2))\cong
\coker d_2:\Gamma(U,{\mathcal{E}}[-1])\to\Gamma(U,{\mathcal{E}}[-3])$.
\item
$\tau^{-1}({\mathbf{R}}^4)={\mathbf{CP}}^3\setminus\{[*,*,0,0]\}
={\mathbf{CP}}^3\setminus L_\infty
={\mathbf{C}}^3\cup{\mathbf{C}}^3$ (two affine patches).
\item $H^2({\mathbf{CP}}^3\setminus L_\infty,{\mathcal{O}}(-2))=0$.
\end{itemize}

Alternatively, as is effectively done in \cite{split}, one can compute the 
$E_2$-differential explicitly to come to the same conclusion.\end{proof}

At this point we have proved Theorem~\ref{HarmonicHull} in case $m=2$, i.e. in
dimension four. Notice already in (\ref{dimensions}) that if this proof is to
extend verbatim to higher dimensions, then $M$ must have even dimension.

\section{Generalities on double fibrations}
Before we generalise the four-dimensional case to higher even dimensions, we
point out two aspects of our discussion so far that apply rather more broadly.

\subsection{Involutive structures in a correspondence} Let us suppose that we 
start with a correspondence of complex manifolds in which one of these 
manifolds is the complexification of a real manifold, as depicted 
in~(\ref{realdiagram}). Letting $F\equiv\nu^{-1}(M)$, we obtain a diagram
\[\raisebox{-20pt}{\begin{picture}(70,50)
\put(35,40){\makebox(0,0){${\mathfrak{X}}$}}
\put(0,5){\makebox(0,0){$Z$}}
\put(70,5){\makebox(0,0){${\mathbf{C}}M$}}
\put(30,35){\vector(-1,-1){25}}
\put(40,35){\vector(1,-1){25}}
\put(75,35){\vector(1,-1){25}}
\put(72,40){\makebox(0,0){$F$}}
\put(107,5){\makebox(0,0){$M$}}
\put(53.5,40){\makebox(0,0){$\supset$}}
\put(90,5){\makebox(0,0){$\supset$}}
\put(67,35){\line(-2,-1){17}}
\put(48,25.5){\vector(-2,-1){41}}
\put(10,20){\makebox(0,0){$\mu$}}
\put(61,20){\makebox(0,0){$\nu$}}
\put(96,20){\makebox(0,0){$\tau$}}
\put(27,20){\makebox(0,0){$\eta$}}
\end{picture}}\]
of which (\ref{specialrealdiagram}) is the special case where it just so
happens that $\eta:F\to Z$ is a diffeomorphism (whilst in general $\eta$ need
be neither injective nor surjective). It is a matter of linear algebra to check
that there is a short exact sequence of vector bundles
\[0\to\Lambda_{\mathfrak{X}}^{1,0}|_F\to\Lambda_F^1\to\Lambda_\tau^{0,1}\to 0\]
on~$F$, where $\Lambda_F^1$ denotes the bundle of ${\mathbf{C}}$-valued
$1$-forms on~$F$. On the other hand, the short exact sequence
\[0\to\mu^*\Lambda_Z^{1,0}\to\Lambda_{\mathfrak{X}}^{1,0}
\to\Lambda_\mu^{1,0}\to 0\]
defines the bundle $\Lambda_\mu^{1,0}$ of forms of type $(1,0)$ along the 
fibres of~$\mu$. Combining these, we obtain a short exact sequence
\begin{equation}\label{masterkey}
0\to\Lambda_\mu^{1,0}|_F\to\Lambda_F^{0,1}\to\Lambda_\tau^{0,1}\to 0
\end{equation}
on~$F$, generalising (\ref{key}), where $\Lambda_F^{0,1}$ is defined as the 
quotient:--
\[0\to\mu^*\Lambda_Z^{1,0}|_F=\eta^*\Lambda_Z^{1,0}
\to\Lambda_F^1\to\Lambda_F^{0,1}\to 0.\]
The operator $\bar\partial:\Lambda_F^0\to\Lambda_F^{0,1}$ defined as the 
composition
\[\Lambda_F^0\stackrel{d}{\longrightarrow}\Lambda_F^1\to\Lambda_F^{0,1}\]
induces $\bar\partial:\Lambda_F^{0,1}\to\Lambda_F^{0,2}
\equiv\Lambda^2(\Lambda_F^{0,1})$ and it is readily verified that
$\bar\partial\,{}^2=0$. A quotient bundle with this property is called an {\it
involutive structure\/} in the sense of Tr\`eves~\cite{bch,t}. Just as
(\ref{key}) gives rise to a spectral sequence for the Dolbeault cohomology of
$Z$ in terms of smooth data on~$M$, so does (\ref{masterkey}) give rise to a
similar spectral sequence for the {\it involutive cohomology\/} of~$F$. This is
a useful procedure whose final conclusion depends on the particular
circumstances and especially how one interprets this involutive cohomology down
on~$Z$.

For the classical correspondence (\ref{twistorcorrespondence}) regarded as
homogeneous under the action of ${\mathrm{SO}}(6,{\mathbf{C}})$ (rather than
its double cover ${\mathrm{SL}}(4,{\mathbf{C}})$), there are three natural 
choices corresponding to the real forms 
$${\mathrm{SO}}(5,1)\qquad{\mathrm{SO}}(4,2)\qquad{\mathrm{SO}}(3,3)$$
in which we take $M$ to be the unique closed orbit in
${\mathrm{Gr}}_2({\mathbf{C}}^4)$ of the real form. 
For ${\mathrm{SO}}(5,1)$ we
obtain $M=S^4$ and the classical Penrose transform of
Theorem~\ref{classicalPenrose}. For ${\mathrm{SO}}(4,2)$ the range of
$\eta:F\to Z$ is the standard indefinite hyperquadric
${\mathcal{N}}\subset{\mathbf{CP}}^3$ and $\eta:F\to{\mathcal{N}}$ is a circle
bundle. Eventually, we obtain by these means a transform on the CR cohomology
of~${\mathcal{N}}$. Finally, for ${\mathrm{SO}}(3,3)$, although the mapping
$\eta:F\to Z$ is a surjection between $6$-dimensional manifolds, it is not a
diffeomorphism. Instead, it is the real blow-up of ${\mathbf{CP}}^3$ along
${\mathbf{RP}}^3$ equipped with the induced involutive
structure~\cite{eastwood_graham}. As explained in~\cite{split}, the involutive
cohomology in this case may be used as a prop in understanding the classical
John transform~\cite{j} on~${\mathbf{R}}^3$.

\subsection{Homogeneous double fibrations} Suppose $G$ is a complex Lie group
with closed Lie subgroups $P$, $Q$, and $R=P\cap Q$. Then there is a double
fibration
\begin{eqnarray}\label{DoubleFibration}
\raisebox{-25pt}{\begin{picture}(80,55)(0,-5)
\put(40,40){\makebox(0,0){$G/R$}}
\put(0,0){\makebox(0,0){$G/Q$}}
\put(80,0){\makebox(0,0){$G/P$}}
\put(30,30){\vector(-1,-1){20}}
\put(50,30){\vector(1,-1){20}}
\put(14,21){\makebox(0,0){$\mu$}}
\put(67,21){\makebox(0,0){$\nu$}}
\end{picture}}\qquad\equiv\qquad
\raisebox{-25pt}{\begin{picture}(80,55)(0,-5)
\put(40,40){\makebox(0,0){${\mathfrak{X}}$}}
\put(0,0){\makebox(0,0){$Z$}}
\put(80,0){\makebox(0,0){${\mathbf{C}}M$}}
\put(30,30){\vector(-1,-1){20}}
\put(50,30){\vector(1,-1){20}}
\put(14,21){\makebox(0,0){$\mu$}}
\put(67,21){\makebox(0,0){$\nu$}}
\end{picture}}
\end{eqnarray}
where we are taking the liberty of writing $G/P$ as
${\mathbf{C}}M$ to indicate that we have in mind this complex manifold being 
the complexification of a real manifold (even though in this subsection we are 
not insisting on this). We shall also suppose that $P/R$ is compact whence the 
fibres of $\nu$ are compact.
  
Take a point $z\in{\mathbf{C}}M$ and consider 
$L_z\equiv\mu(\nu^{-1}(z))\subset Z$. In particular, if $z=o$, the identity
coset in $G/P$, then $L_o=\mu(P/R)$. It is evident that $\mu$ is injective
on $P/R$ and hence $L_o$ is an embedded submanifold of $Z$ isomorphic to $P/R$.
Since the whole double fibration is homogeneous under $G$ it follows that $L_z$
is an embedded submanifold of $Z$ isomorphic to $L_o$ for each
$z\in{\mathbf{C}}M$. The map 
\[\mu\times\nu:G/R\rightarrow G/Q\times G/P\]
realises $G/R$ as the set of pairs $(l,z)\in G/Q\times G/P$ satisfying the
familiar incidence relation $l\in L_z$.
  
We would like to know when $L_z$ intersects $L_{z'}$ in this homogeneous
setting. Let $g_z\in G$ be a representative of $z$ and $g_{z'}$ a
representative of $z'$. Then it is a matter of untangling definitions to check
that $L_z\cap L_{z'}\neq \emptyset$ if and only if
\begin{eqnarray}\label{PQPappears}
g_{z'}^{-1}g_z\in PQP
\end{eqnarray}    
where $PQP$ denotes the set
\begin{eqnarray}PQP=\{pqp':p,p'\in P, q\in Q\}.
\end{eqnarray}
\begin{defn}\label{nullseparation}
Let us write $z\sim z'$ if and only if 
\begin{eqnarray*}
g_{z'}^{-1}g_z\in PQP
\end{eqnarray*} 
for $g_z\in z$ and $g_{z'}\in z'$, noting that this does not depend on choice
of $g_z$ and $g_{z'}$.\end{defn} 
\noindent We have just observed that $z\sim z'$ if and only if $L_z\cap
L_{z'}\not=\emptyset$. 
The relation `$\sim$' is symmetric (but not transitive) and $z\sim z$ always 
holds.
\begin{lem}\label{PQPclosed}
The set $PQP$ is closed in~$G$.
\end{lem} 
\begin{proof}
Let $R$ act on $P\times G$ by $r(p,g)=(pr^{-1},rg)$ and let $P\times_RG$ denote
the quotient. Define $P\times_RQ$ similarly. Then $P\times_RQ$ is a closed
subset of $P\times_RG$. To see this, note that $(p,g)\mapsto(p,pg)$ induces an
isomorphism $P\times_RG\cong (P/R)\times G$. Therefore, if $\{(p_i,q_i)\}$
represents a sequence in $P\times_RQ$ converging in $P\times_R G$, then
$\{(p_iR,p_iq_i)\}$ is a convergent sequence in $(P/R)\times G$ and, by the
definition of quotient topology, so we may modify the sequence $\{p_i\}$
without loss of generality so that it converges in $P$, say to~$p$. It follows
that $\{q_i\}$ converges and, since $Q$ is closed, the limit point $q$ must be
in~$Q$. Therefore, the sequence in $P\times_RQ$ represented by $\{(p_i,q_i)\}$ 
converges to the point represented by $(p,q)$. Similarly, 
\[P\times_RQ\times_RP\subset P\times_RG\times_RP\cong(P/R)\times G\times(P/R)
\quad\mbox{is closed}.\]
For any topological spaces, the projection $X\times K\to X$ is a closed mapping
if $K$ is compact. Therefore the image of $P\times_RQ\times_RP$ under the
projection
\[(P/R)\times G\times(P/R)\to G\]
is closed. This is exactly $PQP$.\end{proof}
\begin{coro}\label{PQP_is_algebraic}
Suppose in addition that all groups involved are algebraic. Then $PQP$ is an
algebraic subvariety of $G$.
\end{coro}
\begin{proof} As the image of $P\times Q\times P$ under the algebraic mapping
of multiplication $G\times G\times G\to G$, it follows that $PQP$ is
constructible~\cite{borel}. Then since $PQP$ is closed in the usual
topology on~$G$, it is also closed in the Zariski topology. Alternatively, one
can repeat the proof of Lemma~\ref{PQPclosed} almost verbatim for the Zariski
topology, using at the end that the image of a closed set under a proper
mapping is closed.
\end{proof}

\section{Harmonic hull in higher even dimensions}
Following Murray~\cite{murray}, to extend the correspondence
(\ref{twistorcorrespondence}) to higher dimensions, we should identify
${\mathrm{Gr}}_2({\mathbf{C}}^4)$ with the non-singular quadric
${\mathbf{Q}}_4\subset{\mathbf{CP}}^5$ as in the proof of 
Proposition~\ref{manufacture_tau} and consider the quadric ${\mathbf{Q}}_{2m}$ 
as a homogeneous space for $G={\mathrm{SO}}(2m+2,{\mathbf{C}})$. It is 
convenient to take 
\[\textstyle\|(x_0,x_1,\ldots,x_m,y_0,y_1,\ldots,y_m)\|^2\equiv
\sum_{j=0}^mx_jy_j\equiv x^ty\]
as quadratic form on ${\mathbf{C}}^{2m+2}$ and hence to realise $G$ as
\[\left\{\begin{pmatrix}
A&B\\
C&D
\end{pmatrix}\mbox{ s.t.\ }\begin{pmatrix}
A&B\\
C&D
\end{pmatrix}^t
\begin{pmatrix}
0 &I\\
I&0
\end{pmatrix}
\begin{pmatrix}
A&B\\
C&D
\end{pmatrix}=\begin{pmatrix} 0&I\\I&0\end{pmatrix}\right\}
\]
where all these matrices are written in $(m+1)\times(m+1)$ blocks. In other 
words, these blocks are constrained by the following relations.
\begin{equation}\label{constraints}
A^t C+C^t A=0\qquad A^tD+C^tB=I\qquad B^tD+D^t B=0.\end{equation}
Since with these conventions the first standard basis vector in
${\mathbf{C}}^{2m+2}$ is null, we have ${\mathbf{Q}}_{2m}=G/P$, where 
\[P=\left\{\mbox{\scriptsize$\begin{pmatrix}*&*&\cdots&*\\
0&*&\cdots&*\\
\vdots&\vdots&\ddots&\vdots\\
0&*&\cdots&*
\end{pmatrix}$}\in{\mathrm{SO}}(2m+2,{\mathbf{C}})\right\}.\]
More explicitly, it is easy to check that elements of $P$ may be written
uniquely as 
\begin{equation}\label{P}\begin{pmatrix}\lambda&0&0&0\\
0&\hat A&0&\hat B\\
0&0&\lambda^{-1}&0\\
0&\hat C&0&\hat D\end{pmatrix}
\begin{pmatrix}1&-q^t&-p^tq&-p^t\\
0&I&p&0\\
0&0&1&0\\
0&0&q&I\end{pmatrix}\end{equation}
where these blocks are of appropriate sizes and such that 
\[\begin{pmatrix}\hat A&\hat B\\
\hat C&\hat D\end{pmatrix}\in{\mathrm{SO}}(2m,{\mathbf{C}}).\]
For the twistor space in this case, we shall take
\begin{equation}\label{Q}
Z=G/Q\quad\mbox{where }Q=\left\{\begin{pmatrix}A&B\\ 0&D\end{pmatrix}
\in{\mathrm{SO}}(2m+2,{\mathbf{C}})\right\}.\end{equation}
Notice that (\ref{constraints}) forces $D=(A^t)^{-1}$ and $B=AE$, where $E$ is 
skew. Therefore $Q$ has dimension $(m+1)^2+m(m+1)/2=(3m+2)(m+1)/2$ whence 
\[\dim_{\mathbf{C}}Z=(m+1)(2m+1)-(3m+2)(m+1)/2=m(m+1)/2\]
and in comparison with (\ref{dimensions}) we see that $s=m(m-1)/2$. Setting 
${\mathfrak{X}}=G/(P\cap Q)$ and using the notation of~\cite{beastwood}, we 
obtain the following correspondence
\[\raisebox{-25pt}{\begin{picture}(80,55)(0,-5)
\put(40,40){\makebox(0,0){${\mathfrak{X}}$}}
\put(0,0){\makebox(0,0){$Z$}}
\put(80,0){\makebox(0,0){${\mathbf{C}}M$}}
\put(30,30){\vector(-1,-1){20}}
\put(50,30){\vector(1,-1){20}}
\put(14,21){\makebox(0,0){$\mu$}}
\put(67,21){\makebox(0,0){$\nu$}}
\end{picture}}\qquad=\qquad\qquad
\raisebox{-25pt}{\begin{picture}(80,55)(0,-5)
\put(40,40){\makebox(0,0){
\begin{picture}(50,20)
\put(0,10){\makebox(0,0){$\cross$}}
\put(0,10){\line(1,0){40}}
\put(10,10){\makebox(0,0){$\bullet$}}
\put(20,10){\makebox(0,0){$\bullet$}}
\put(30,10){\makebox(0,0){$\bullet$}}
\put(40,10){\makebox(0,0){$\bullet$}}
\put(40,10){\line(2,1){10}}
\put(40,10){\line(2,-1){10}}
\put(50,15){\makebox(0,0){$\times$}}
\put(50,5){\makebox(0,0){$\bullet$}}
\end{picture}}}
\put(0,0){\makebox(0,0){
\begin{picture}(50,20)
\put(0,10){\makebox(0,0){$\bullet$}}
\put(0,10){\line(1,0){40}}
\put(10,10){\makebox(0,0){$\bullet$}}
\put(20,10){\makebox(0,0){$\bullet$}}
\put(30,10){\makebox(0,0){$\bullet$}}
\put(40,10){\makebox(0,0){$\bullet$}}
\put(40,10){\line(2,1){10}}
\put(40,10){\line(2,-1){10}}
\put(50,15){\makebox(0,0){$\times$}}
\put(50,5){\makebox(0,0){$\bullet$}}
\end{picture}}}
\put(80,0){\makebox(0,0){
\begin{picture}(50,20)
\put(0,10){\makebox(0,0){$\cross$}}
\put(0,10){\line(1,0){40}}
\put(10,10){\makebox(0,0){$\bullet$}}
\put(20,10){\makebox(0,0){$\bullet$}}
\put(30,10){\makebox(0,0){$\bullet$}}
\put(40,10){\makebox(0,0){$\bullet$}}
\put(40,10){\line(2,1){10}}
\put(40,10){\line(2,-1){10}}
\put(50,15){\makebox(0,0){$\bullet$}}
\put(50,5){\makebox(0,0){$\bullet$}}
\end{picture}}}
\put(30,30){\vector(-1,-1){20}}
\put(50,30){\vector(1,-1){20}}
\put(14,21){\makebox(0,0){$\mu$}}
\put(67,21){\makebox(0,0){$\nu$}}
\end{picture}}\]
(for convenience, written in case $m=6$) as detailed
in~\cite[pp.~111--115]{beastwood}.

\begin{lemma}
$PQP\subset{\mathrm{SO}}(2m+2,{\mathbf{C}})$ is defined by the equation 
$C_{11}=0$, i.e.
\begin{eqnarray}\label{PQP}
PQP=\left\{\begin{pmatrix}
\begin{picture}(50,50)(-2.5,0)
\thinlines
\put(-5,22.7){\line(1,0){55}}
\put(22.5,-5){\line(0,1){55}}
\put(0,45){\makebox(0,0){$\ast$}}
\put(5,45){\makebox(0,0){$\ast$}}
\put(10,45){\makebox(0,0){$.$}}
\put(15,45){\makebox(0,0){$.$}}
\put(20,45){\makebox(0,0){$\ast$}}
\put(25,45){\makebox(0,0){$\ast$}}
\put(30,45){\makebox(0,0){$\ast$}}
\put(35,45){\makebox(0,0){$.$}}
\put(40,45){\makebox(0,0){$.$}}
\put(45,45){\makebox(0,0){$\ast$}}
\put(0,40){\makebox(0,0){$\ast$}}
\put(5,40){\makebox(0,0){$\ast$}}
\put(10,40){\makebox(0,0){$.$}}
\put(15,40){\makebox(0,0){$.$}}
\put(20,40){\makebox(0,0){$\ast$}}
\put(25,40){\makebox(0,0){$\ast$}}
\put(30,40){\makebox(0,0){$\ast$}}
\put(35,40){\makebox(0,0){$.$}}
\put(40,40){\makebox(0,0){$.$}}
\put(45,40){\makebox(0,0){$\ast$}}
\put(0,35){\makebox(0,0){$.$}}
\put(5,35){\makebox(0,0){$.$}}
\put(10,35){\makebox(0,0){$.$}}
\put(15,35){\makebox(0,0){$.$}}
\put(20,35){\makebox(0,0){$.$}}
\put(25,35){\makebox(0,0){$.$}}
\put(30,35){\makebox(0,0){$.$}}
\put(35,35){\makebox(0,0){$.$}}
\put(40,35){\makebox(0,0){$.$}}
\put(45,35){\makebox(0,0){$.$}}
\put(0,30){\makebox(0,0){$.$}}
\put(5,30){\makebox(0,0){$.$}}
\put(10,30){\makebox(0,0){$.$}}
\put(15,30){\makebox(0,0){$.$}}
\put(20,30){\makebox(0,0){$.$}}
\put(25,30){\makebox(0,0){$.$}}
\put(30,30){\makebox(0,0){$.$}}
\put(35,30){\makebox(0,0){$.$}}
\put(40,30){\makebox(0,0){$.$}}
\put(45,30){\makebox(0,0){$.$}}
\put(0,25){\makebox(0,0){$\ast$}}
\put(5,25){\makebox(0,0){$\ast$}}
\put(10,25){\makebox(0,0){$.$}}
\put(15,25){\makebox(0,0){$.$}}
\put(20,25){\makebox(0,0){$\ast$}}
\put(25,25){\makebox(0,0){$\ast$}}
\put(30,25){\makebox(0,0){$\ast$}}
\put(35,25){\makebox(0,0){$.$}}
\put(40,25){\makebox(0,0){$.$}}
\put(45,25){\makebox(0,0){$\ast$}}
\put(0,20){\makebox(0,0){\scriptsize$0$}}
\put(5,20){\makebox(0,0){$\ast$}}
\put(10,20){\makebox(0,0){$.$}}
\put(15,20){\makebox(0,0){$.$}}
\put(20,20){\makebox(0,0){$\ast$}}
\put(25,20){\makebox(0,0){$\ast$}}
\put(30,20){\makebox(0,0){$\ast$}}
\put(35,20){\makebox(0,0){$.$}}
\put(40,20){\makebox(0,0){$.$}}
\put(45,20){\makebox(0,0){$\ast$}}
\put(0,15){\makebox(0,0){$\ast$}}
\put(5,15){\makebox(0,0){$\ast$}}
\put(10,15){\makebox(0,0){$.$}}
\put(15,15){\makebox(0,0){$.$}}
\put(20,15){\makebox(0,0){$\ast$}}
\put(25,15){\makebox(0,0){$\ast$}}
\put(30,15){\makebox(0,0){$\ast$}}
\put(35,15){\makebox(0,0){$.$}}
\put(40,15){\makebox(0,0){$.$}}
\put(45,15){\makebox(0,0){$\ast$}}
\put(0,10){\makebox(0,0){$.$}}
\put(5,10){\makebox(0,0){$.$}}
\put(10,10){\makebox(0,0){$.$}}
\put(15,10){\makebox(0,0){$.$}}
\put(20,10){\makebox(0,0){$.$}}
\put(25,10){\makebox(0,0){$.$}}
\put(30,10){\makebox(0,0){$.$}}
\put(35,10){\makebox(0,0){$.$}}
\put(40,10){\makebox(0,0){$.$}}
\put(45,10){\makebox(0,0){$.$}}
\put(0,5){\makebox(0,0){$.$}}
\put(5,5){\makebox(0,0){$.$}}
\put(10,5){\makebox(0,0){$.$}}
\put(15,5){\makebox(0,0){$.$}}
\put(20,5){\makebox(0,0){$.$}}
\put(25,5){\makebox(0,0){$.$}}
\put(30,5){\makebox(0,0){$.$}}
\put(35,5){\makebox(0,0){$.$}}
\put(40,5){\makebox(0,0){$.$}}
\put(45,5){\makebox(0,0){$.$}}
\put(0,0){\makebox(0,0){$\ast$}}
\put(5,0){\makebox(0,0){$\ast$}}
\put(10,0){\makebox(0,0){$.$}}
\put(15,0){\makebox(0,0){$.$}}
\put(20,0){\makebox(0,0){$\ast$}}
\put(25,0){\makebox(0,0){$\ast$}}
\put(30,0){\makebox(0,0){$\ast$}}
\put(35,0){\makebox(0,0){$.$}}
\put(40,0){\makebox(0,0){$.$}}
\put(45,0){\makebox(0,0){$\ast$}}
\end{picture}
\end{pmatrix}\in SO(2m+2,{\mathbf{C}})\right\}.\end{eqnarray}
\end{lemma} 
\begin{proof}
By applying $PQP$ to the first standard basis vector, bearing in mind the form
(\ref{P}) and (\ref{Q}) of elements from $P$ and $Q$, it is clear that
\begin{equation}\label{clearthat}PQP\subseteq\{C_{11}=0\}.\end{equation}
It is certainly possible to show equality by direct calculation and
normalisation (using that there are just two orbits for the action of $G$ on
${\mathbf{C}}^{2m+2}\setminus\{0\}$ according to whether a vector is null or
not). Alternatively, we can use Corollary~\ref{PQP_is_algebraic} to infer
equality once we know that they have the same dimension provided that the 
right hand side is also irreducible. To compute the dimension of~$PQP$, let us 
take $z'=o$ in (\ref{PQPappears}) to conclude that 
\[g_z\in PQP\iff L_z\cap L_o\not=\emptyset.\]
Therefore, 
\begin{equation}\label{dimPQP}\begin{array}{rcl}\dim PQP&=&
\dim P+\dim\{z\in{\mathbf{C}}M\mbox{ s.t.\ }L_z\cap L_o\not=\emptyset\}\\
&=&2m^2+m+1
+\dim\{z\in{\mathbf{C}}M\mbox{ s.t.\ }L_z\cap L_o\not=\emptyset\}.
\end{array}\end{equation}
Now we need to know $\dim(L_z\cap L_o)$ in case this intersection is non-empty. 
For this we can take $z$ to be represented by the second standard basis vector 
and identify this intersection with 
\[{\mathrm{SO}}(2m-2,{\mathbf{C}})\Big/\left\{\begin{pmatrix}
\tilde A&\tilde A\tilde E\\
0&(\tilde A^t)^{-1} \end{pmatrix}\right\},\] 
deducing that it has dimension $(m-1)(m-2)/2$. Bearing in mind that the fibres
of $\mu$ have dimension $m$ and that $L_o$ itself has dimension $m(m-1)/2$ we
find that 
\[\dim\{z\in{\mathbf{C}}M\mbox{ s.t.\ }L_z\cap L_o\not=\emptyset\}=
m+\frac{m(m-1)}{2}-\frac{(m-1)(m-2)}{2}=2m-1.\]
{From} (\ref{dimPQP}) we deduce that 
\[\dim PQP=2m^2+m+1+(2m-1)=m(2m+3).\]
Therefore, $PQP$ has codimension $1$ in ${\mathrm{SO}}(2m+2,{\mathbf{C}})$ and
it remains to show that the right hand side of (\ref{clearthat}) is
irreducible. For this one checks in local co\"ordinates that
$\{C_{11}=0\}\subset{\mathrm{SO}}(2m+2,{\mathbf{C}})$ is smooth except along
$P$ where it has a quadratic singularity. Since $P$ has codimension $2m$ in $G$
and $m\geq 2$ irreducibility follows.
\end{proof}

To complete the proof of Theorem~\ref{HarmonicHull} in higher dimensions, there
are just two remaining issues to address. The first is geometric, namely to
check that the relation $z\sim z'$ of Definition~\ref{nullseparation}
corresponds to null separation as it does when $m=2$. The second issue is 
analytic, namely to check that there is an appropriate double fibration 
transform. 

To understand the geometry one works in an affine chart on 
${\mathbf{C}}M=G/P$. Using block matrices as above, it is easy to check that 
\[{\mathbf{C}}^{2m}\ni(x,y)\longmapsto\begin{pmatrix}1&0&0&0\\
x&I&0&0\\
-x^ty&-y^t&1&-x^t\\
y&0&0&I\end{pmatrix}\]
defines a c\"oordinate chart on $G/P$ in which vector addition corresponds 
exactly to multiplication of such matrices. In particular, 
\[(x,y)\sim(x',y')\iff(x-x')^t(y-y')=0,\]
which is the null separation we require. It is also straightforward to check
in these co\"ordinates that the round sphere $S^{2m}$ is conformally
embedded as a totally real submanifold of~${\mathbf{C}}M$.

Finally, it remains to generalise the Penrose transform to this setting and 
for this we follow the proof of Theorem~\ref{classicalPenrose}. The essential 
point is that this proof reduces the spectral sequence for the fibration
\[\tau:Z\to S^{2m}\]
to the spectral sequence of the more familiar double fibration transform for
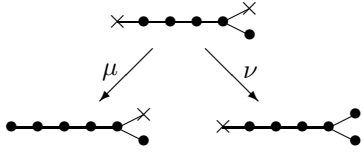
\[\begin{picture}(80,55)(0,-5)
\put(40,40){\makebox(0,0){
\begin{picture}(50,20)
\put(0,10){\makebox(0,0){$\cross$}}
\put(0,10){\line(1,0){40}}
\put(10,10){\makebox(0,0){$\bullet$}}
\put(20,10){\makebox(0,0){$\bullet$}}
\put(30,10){\makebox(0,0){$\bullet$}}
\put(40,10){\makebox(0,0){$\bullet$}}
\put(40,10){\line(2,1){10}}
\put(40,10){\line(2,-1){10}}
\put(50,15){\makebox(0,0){$\times$}}
\put(50,5){\makebox(0,0){$\bullet$}}
\end{picture}}}
\put(0,0){\makebox(0,0){
\begin{picture}(50,20)
\put(0,10){\makebox(0,0){$\bullet$}}
\put(0,10){\line(1,0){40}}
\put(10,10){\makebox(0,0){$\bullet$}}
\put(20,10){\makebox(0,0){$\bullet$}}
\put(30,10){\makebox(0,0){$\bullet$}}
\put(40,10){\makebox(0,0){$\bullet$}}
\put(40,10){\line(2,1){10}}
\put(40,10){\line(2,-1){10}}
\put(50,15){\makebox(0,0){$\times$}}
\put(50,5){\makebox(0,0){$\bullet$}}
\end{picture}}}
\put(80,0){\makebox(0,0){
\begin{picture}(50,20)
\put(0,10){\makebox(0,0){$\cross$}}
\put(0,10){\line(1,0){40}}
\put(10,10){\makebox(0,0){$\bullet$}}
\put(20,10){\makebox(0,0){$\bullet$}}
\put(30,10){\makebox(0,0){$\bullet$}}
\put(40,10){\makebox(0,0){$\bullet$}}
\put(40,10){\line(2,1){10}}
\put(40,10){\line(2,-1){10}}
\put(50,15){\makebox(0,0){$\bullet$}}
\put(50,5){\makebox(0,0){$\bullet$}}
\end{picture}}}
\put(30,30){\vector(-1,-1){20}}
\put(50,30){\vector(1,-1){20}}
\put(14,21){\makebox(0,0){$\mu$}}
\put(67,21){\makebox(0,0){$\nu$}}
\end{picture}\]
and this has already been worked out~\cite[p.~115]{beastwood}. We obtain the 
following result (as in~\cite[Theorem~32]{murray}).
\begin{theo}
For $U\subseteq{\mathbf{R}}^{2m}$ any open subset, there is a 
natural isomorphism
\[{\mathcal{P}}:H^{m(m-1)/2}(\tau^{-1}(U),{\mathcal{O}}(-2m+2))
\stackrel{\;\simeq\quad}{\longrightarrow}
\{\phi:U\to{\mathbf{C}}\mbox{ s.t.\ }\Delta\phi=0\}.\]
\end{theo}
\noindent Theorem~\ref{HarmonicHull} is an immediate consequence.

\section{Harmonic hull in odd dimensions}\label{odd}
The harmonic hull in odd dimensions behaves differently. The most blatant
difference is that it may no longer exist in the na\"{\i}ve sense defined in
the introduction. Let us see, for example, that there is no
maximal open set in ${\mathbf{C}}^3$ to which all harmonic functions on
$U={\mathbf{R}}^3\setminus\{0\}$ extend. In accordance with
Definition~\ref{tildeU}
\begin{equation}\label{tildeUinthiscase}
\tilde{U}={\mathbf{C}}^3\setminus\{z_1{}^2+z_2{}^2+z_3{}^2=0\}.\end{equation}
We shall see that $\tilde{U}$ is the only candidate for the harmonic hull 
and yet there are harmonic functions on $U$ that do not extend there. The 
Newtonian potential 
\[r(x)\equiv\frac{1}{\sqrt{x_1{}^2+x_2{}^2+x_3{}^2}}\]
is harmonic on~$U$ and extends to a neighbourhood of $U$ in ${\mathbf{C}}^3$ as
\[r(z)\equiv\frac{1}{\sqrt{z_1{}^2+z_2{}^2+z_3{}^2}}\]
for a suitably chosen branch of square root. Consider the embedding
\[F:{\mathbf{C}}\setminus\{-i,0,i\}\hookrightarrow\tilde{U}\]
given by $F(\zeta)=(\zeta,\zeta^2,0)$. There is no well-defined branch of
\[r\circ F(\zeta)=\frac{1}{\zeta\sqrt{1+\zeta^2}}\]
on ${\mathbf{C}}\setminus\{-i,0,i\}$ and so there is no well-defined branch of 
$r(z)$ on~$\tilde{U}$. Nevertheless, harmonic functions on $U$ do extend into 
${\mathbf{C}}^3$ and, in fact, it is shown in \cite{ACL} that all harmonic 
functions $U$ extend uniquely to the {\it reduced\/} harmonic hull
\begin{equation}\label{reducedhull}
{\mathbf{C}}^3\setminus\{z_1{}^2+z_2{}^2+z_3{}^2\in{\mathbf{R}}_{\leq 0}\}.
\end{equation}
We shall use reduced harmonic hulls together with the conformal invariance of
the Laplacian to write $\tilde{U}$ as a union of open subsets of
${\mathbf{C}}^3$ to which all harmonic functions on $U$ extend. By doing so,
we see that $U$ cannot have a harmonic hull.

If $f$ is harmonic on an open subset of ${\mathbf{R}}^3$, then 
\[F({\mathbf{X}})\equiv\frac{\displaystyle f\Big(
\frac{X_1-\epsilon\|{\mathbf{X}}\|^2}
{1-2\epsilon X_1+\epsilon^2\|{\mathbf{X}}\|^2},
\frac{X_2}{1-2\epsilon X_1+\epsilon^2\|{\mathbf{X}}\|^2},
\frac{X_3}{1-2\epsilon X_1+\epsilon^2\|{\mathbf{X}}\|^2}\Big)}
{\sqrt{1-2\epsilon X_1+\epsilon^2\|{\mathbf{X}}\|^2}}
\]
is harmonic wherever $1-2\epsilon X_1+\epsilon^2\|{\mathbf{X}}\|^2>0$. 
But
\[1-2\epsilon X_1+\epsilon^2\|{\mathbf{X}}\|^2=
\epsilon^2\big((X_1-1/\epsilon)^2+X_2{}^2+X_3{}^2\big),\]
so if $f$ is harmonic on $U$, then $F$ is harmonic on
${\mathbf{R}}^3\setminus\{(1/\epsilon,0,0),(0,0,0)\}$. {From} the conformal
point of view $f$ and $F$ are the same conformal density of weight $-1/2$
defined on the twice-punctured sphere but viewed in two different flat
co\"ordinate systems via stereographic projection. In any case, by
\cite{ACL} the harmonic function $F$ extends to the reduced harmonic hull of
${\mathbf{R}}^3\setminus\{(1/\epsilon,0,0),(0,0,0)\}$, namely
\begin{equation}\label{changeandextend}
{\mathbf{C}}^3\setminus\big(\{Z_1{}^2+Z_2{}^2+Z_3{}^2\in{\mathbf{R}}_{\leq 0}\}
\cup\{(Z_1-1/\epsilon)^2+Z_2{}^2+Z_3{}^2\in{\mathbf{R}}_{\leq 0}\}\big).
\end{equation}
But now if $F$ is holomorphic and complex harmonic on an open subset of 
${\mathbf{C}}^3$, and we write ${\mathbf{z}}^2$ for $z_1{}^2+z_2{}^2+z_3{}^2$, 
then 
\begin{equation}\label{finaldestination}
f({\mathbf{z}})\equiv\frac{\displaystyle F\Big(
\frac{z_1+\epsilon{\mathbf{z}}^2}
{1+2\epsilon z_1+\epsilon^2{\mathbf{z}}^2},
\frac{z_2}{1+2\epsilon z_1+\epsilon^2{\mathbf{z}}^2},
\frac{z_3}{1+2\epsilon z_1+\epsilon^2{\mathbf{z}}^2}\Big)}
{\sqrt{1+2\epsilon z_1+\epsilon^2{\mathbf{z}}^2}}\end{equation}
is harmonic wherever one can choose a well-defined branch 
of~$\sqrt{1+2\epsilon z_1+\epsilon^2{\mathbf{z}}^2}$. In particular, if we 
insist that $\|{\mathbf{z}}\|^2<1/(9\epsilon^2)$, then 
\[\begin{array}{rcl}\Re(1+2\epsilon z_1+\epsilon^2{\mathbf{z}}^2)&=&
(1+\epsilon x_1)^2+\epsilon^2(x_2{}^2+x_3{}^2-y_1{}^2-y_2{}^2-y_3{}^2)\\
&>&(2/3)^2-1/9-1/9-1/9\enskip=\enskip 1/9\enskip>\enskip 0
\end{array}\]
whence there is certainly no problem in
defining~$\sqrt{1+2\epsilon z_1+\epsilon^2{\mathbf{z}}^2}$. Now
\[(Z_1,Z_2,Z_3)=\Big(
\frac{z_1+\epsilon{\mathbf{z}}^2}
{1+2\epsilon z_1+\epsilon^2{\mathbf{z}}^2},
\frac{z_2}{1+2\epsilon z_1+\epsilon^2{\mathbf{z}}^2},
\frac{z_3}{1+2\epsilon z_1+\epsilon^2{\mathbf{z}}^2}\Big)\]
if and only if 
\[(z_1,z_2,z_3)=\Big(
\frac{Z_1-\epsilon{\mathbf{Z}}^2}
{1-2\epsilon Z_1+\epsilon^2{\mathbf{Z}}^2},
\frac{Z_2}{1-2\epsilon Z_1+\epsilon^2{\mathbf{Z}}^2},
\frac{Z_3}{1-2\epsilon Z_1+\epsilon^2{\mathbf{Z}}^2}\Big)\]
in which case
\[1=(1+2\epsilon z_1+\epsilon^2{\mathbf{z}}^2)
(1-2\epsilon Z_1+\epsilon^2{\mathbf{Z}}^2).\]
Therefore, by changing co\"ordinates on $U$ to obtain $F({\mathbf{X}})$ from 
$f({\mathbf{x}})$, then extending to the reduced harmonic hull 
(\ref{changeandextend}), and finally considering $f({\mathbf{z}})$ defined 
by~(\ref{finaldestination}), we certainly obtain a holomorphic extension of 
$f({\mathbf{x}})|_{\{\|{\mathbf{x}}\|^2<1/(9\epsilon^2)\}}$ to 
\begin{equation}\label{extended}
\left\{{\mathbf{z}}\in{\mathbf{C}}^3\mbox{ s.t.}
\begin{array}l\|{\mathbf{z}}\|^2<1/(9\epsilon^2)\qquad
{\mathbf{Z}}^2\not\in{\mathbf{R}}_{\leq 0}\\
(Z_1-1/\epsilon)^2+Z_2{}^2+Z_3{}^2\not\in{\mathbf{R}}_{\leq 0}
\end{array}\!\right\}\end{equation}
But, using the identities
\[{\mathbf{Z}}^2=
\frac{{\mathbf{z}}^2}{1+2\epsilon z_1+\epsilon^2{\mathbf{z}}^2}
\qquad (Z_1-1/\epsilon)^2+Z_2{}^2+Z_3{}^2=
\frac{1}{\epsilon^2(1+2\epsilon z_1+\epsilon^2{\mathbf{z}}^2)}\]
and our previous observation that 
$\Re(1+2\epsilon z_1+\epsilon^2{\mathbf{z}}^2)>0$ when 
$\|{\mathbf{z}}\|^2<1/(9\epsilon^2)$, we may rewrite 
(\ref{extended}) as
\begin{equation}\label{curvedextension}
\left\{{\mathbf{z}}\in{\mathbf{C}}^3\mbox{ s.t.\ }
\|{\mathbf{z}}\|^2<1/(9\epsilon^2)\mbox{ and }
\frac{{\mathbf{z}}^2}{1+2\epsilon z_1+\epsilon^2{\mathbf{z}}^2}
\not\in{\mathbf{R}}_{\leq 0}\right\}.\end{equation}
We claim that all points in $\tilde{U}$ (given explicitly 
as~(\ref{tildeUinthiscase})) that 
are not in~(\ref{reducedhull}), lie in (\ref{curvedextension}) for some choice 
of orthogonal co\"ordinate system on ${\mathbf{R}}^3$ and 
some~$\epsilon\not=0$. In other words, we should accomplish this for 
\[{\mathbf{z}}={\mathbf{x}}+i{\mathbf{y}}\in{\mathbf{C}}^3\mbox{ s.t.\ }
\langle x,y\rangle=0\mbox{ and }\|x\|^2<\|y\|^2.\]
Choose co\"ordinates so that $y=(y_1,0,0)$. It follows that $x_1=0$ and 
$y_1\not=0$ whence
\[\frac{{\mathbf{z}}^2}{1+2\epsilon z_1+\epsilon^2{\mathbf{z}}^2}=
\frac{{\mathbf{z}}^2}{1+2i\epsilon y_1+\epsilon^2{\mathbf{z}}^2}\]
cannot be real when ${\mathbf{z}}^2$ is real. Hence ${\mathbf{z}}$ 
lies in the set (\ref{curvedextension}) if~$\|z\|^2<1/(9\epsilon^2)$, a 
condition which evaporates as $\epsilon\to 0$. 
We have shown that 
\[\tilde{U}=\bigcup_{A\in{\mathrm{SO}}(3)}\bigcup_{\epsilon\in{\mathbf{R}}}
A\Big\{{\mathbf{z}}\in{\mathbf{C}}^3\mbox{ s.t.\ }
\|{\mathbf{z}}\|^2<1/(9\epsilon^2)\mbox{ and } 
\frac{{\mathbf{z}}^2}{1+2\epsilon z_1+\epsilon^2{\mathbf{z}}^2}
\not\in{\mathbf{R}}_{\leq 0}\Big\}\]
completing our claim that $\tilde{U}$ may be written as a union of open subsets
of ${\mathbf{C}}^3$ to which all harmonic functions on $U$ extend. The argument
immediately generalises to show that ${\mathbf{R}}^n\setminus\{0\}$ does not
have harmonic hull for any odd~$n$.

\end{document}